\documentclass[12pt, a4paper]{amsart}
\usepackage{mathtools, amssymb, amsthm, xcolor}
\usepackage[hidelinks]{hyperref}
\usepackage[centering, heightrounded]{geometry}
\renewcommand{\Re}{\operatorname{Re}}
\renewcommand{\Im}{\operatorname{Im}}
\numberwithin{equation}{section}

\newtheorem{theorem}{Theorem}[section]
\newtheorem{lemma}[theorem]{Lemma}
\newtheorem{proposition}[theorem]{Proposition}

\theoremstyle{definition}

\theoremstyle{remark}
\newtheorem{remark}[theorem]{Remark}

\title[Ground states with point interaction]{Uniqueness and nondegeneracy of ground states for 2d-nonlinear scalar field equations with point interaction}

\author[N. Fukaya]{Noriyoshi Fukaya}
\address{Waseda Research Institute for Science and Engineering, 
Waseda University, 
Tokyo 169-8555, Japan
\newline\indent
Osaka Central Advanced Mathematical Institute, 
Osaka Metropolitan University, 
Osaka 558-8585, Japan
\newline\indent
Institute for Mathematics and Computer Science, 
Tsuda University, 
Tokyo 187-8577, Japan}
\email{nfukaya@aoni.waseda.jp}

\date{
Updated at \today}
\newcommand{\R}{\mathbb{R}}

\newcommand{\al}{\alpha}

\newcommand{\lam}{\lambda}
\newcommand{\rad}{\mathrm{r}}
\newcommand{\set}[3][]{#1\{#2\;\!#1|\;#3#1\}}
\newcommand{\bset}[2]{\left\{#1\;\!\middle|\;#2\right\}}

\begin{document}

\begin{abstract}
We study uniqueness and nondegeneracy of ground states for nonlinear scalar field equations in two dimensions with a point interaction at the origin. It is known that the all ground states are radial, positive, and decreasing functions. In this paper we prove the uniqueness of positive radial solutions by a method of Poho\v{z}aev identities. As a corollary, we obtain the uniqueness of ground states. Moreover, by a variational and ODE technique, we show that the ground state is a nondegenerate critical point of the action in the energy space.
\end{abstract}

\maketitle

\tableofcontents

\section{Introduction}

We consider the following nonlinear scalar field equation in two dimensions:
\begin{equation}\label{eq:sp}
	(-\Delta_\alpha 
	+\lambda )u
	-|u|^{p-1}u
	=0,\quad 
	u\in H_\alpha^1(\R^2),	
\end{equation}
and the following boundary value problem for the ordinary differential equation corresponding to positive radial solutions for \eqref{eq:sp}:
\begin{equation} \label{eq:ODE}
    \left\{\begin{alignedat}{2}
   &\mathopen{}
    - u''
    -\frac{1}{r} u' 
    +\lambda u 
    -u^p 
    =0,& \quad
   &r>0,
\\ &\mathopen{}
    u(r)
    =y_u\Bigl(-\frac{1}{2\pi}\log r
    +\alpha\Bigr)+o(1)&\quad
   &\text{as }r\downarrow 0
    \text{ for some }y_u>0,
\\ &\mathopen{}u(r)\to 0&\quad
   &\text{as }r\to\infty.
    \end{alignedat}\right.
\end{equation}
Here $p>1$, $\lambda>0$, $-\Delta_\alpha$ is the Schr\"odinger operator with a point interaction at the origin with its interaction strength $\alpha\in\R$, and $H_\alpha^1(\R^2)$ is the form domain (or the energy space) of $-\Delta_\alpha$. The operator $-\Delta_\alpha $ is realized as a self-adjoint extension of the symmetric operator
\begin{equation}  \label{eq:op0}
    \left\{\begin{aligned}
   &D(H_0)=C_c^\infty(\R^2\setminus\{0\}),
\\ &H_0=-\Delta
    \end{aligned}\right.
\end{equation} 
in the Hilbert space $L^2(\R^2)$ (see \cite{AlbeGeszHo-eHold05} and references therein for the comprehensive overview) and formally expressed with the Dirac delta function $\delta_0$ centered at the origin as 
\[  -\Delta_\alpha v
    =-\Delta v
    +C(\alpha, v)\delta_0,\quad 
    v\in D(-\Delta_\alpha) \] 
for some constant $C(\alpha, v)\in \mathbb{C}$ (see \eqref{eq:defpi} and after that for the precise definition of $-\Delta_\alpha$). 

The equation \eqref{eq:sp} arises when considering the standing wave solution $\psi(t, x) = e^{i\lambda t}u(x)$ for the nonlinear Schr\"odinger equation with a point interaction:
\begin{equation} \label{eq:NLS}
i\partial_t\psi = -\Delta_\alpha\psi - |\psi|^{p-1}\psi,\quad (t, x) \in \mathbb{R} \times \mathbb{R}^2.
\end{equation}
Recently, the equation \eqref{eq:NLS} has garnered significant attention and has been extensively studied. For physical background,  see \cite{SakaMalo20, ShamChenMaloSaka20}. The local well-posedness of the Cauchy problem for \eqref{eq:NLS} was established in the operator domain $D(-\Delta_\alpha)$ by \cite{CaccFincNoja21} and in the form domain $H_\alpha^1(\R^2)$ by \cite{FukaGeorIked22}, utilizing the Strichartz estimates in \cite{CornMichYaji19}. Further research has focused on the existence, symmetry, and stability of standing waves with minimal action, known as ground states, as investigated in \cite{AdamBoniCarlTent22-2d, FukaGeorIked22}. The (non-)scattering of solutions was explored in \cite{CaccFincNoja23}, while \cite{FincNoja23} studied the blowup of solutions and the strong instability of standing waves. For related problems in three dimensions or involving different nonlinearities, see, for example, \cite{AdamBoniCarlTent22-3d, GeorMichScan24, MichOlgiScan18, PompWata}. The equation \eqref{eq:sp} is also connected to singular solutions of the nonlinear scalar field equation:
\begin{equation}\label{eq:spwop}
(-\Delta + \lambda) u - |u|^{p-1}u = 0,
\end{equation}
because solutions of \eqref{eq:sp} satisfy \eqref{eq:spwop} for all $x \in \mathbb{R}^2 \setminus \{0\}$ along with the boundary conditions:
\begin{align*}
&\lim_{|x| \to 0} |u(x)| = \infty,&
&\lim_{|x| \to \infty} u(x) = 0.
\end{align*}
The study of singular solutions has been developed in different literature; see, for example, \cite{BrezVero80, JohnPanYi94, Lion80, VazqVero84, Vero81}. This paper aims to establish the uniqueness and nondegeneracy of ground states for \eqref{eq:sp}. These properties are crucial for analysing the dynamics around the standing waves for \eqref{eq:NLS}, for example, when applying results from \cite{CazeLion82, GrilShatStra87, Shat83, ShatStra85, Wein86}.

Now following \cite{AlbeGeszHo-eHold05} we briefly state some properties on the Schr\"odinger operator with a point interaction at the origin. It is known that all self-adjoint extensions of \eqref{eq:op0} is expressed as the one-parameter family $(-\Delta_\alpha)_{\alpha\in(-\infty, \infty]}$ of self-adjoint operators $-\Delta_\alpha$ defined by
\begin{equation}\label{eq:defpi}
    \left\{\begin{aligned}
   &D(-\Delta_\alpha)
    =\set[\Big]{f+\dfrac{f(0)}{\alpha+\beta(\lambda)}G_\lambda}{f\in H^2(\R^2)},
\\ &(-\Delta_\alpha+\lambda)v
    =(-\Delta+\lambda)f,\quad v=f+\dfrac{f(0)}{\alpha+\beta(\lambda)}\in D(-\Delta_\alpha)
    \end{aligned}\right.
\end{equation}
for $\alpha\in\R$ and 
\[  D(-\Delta_\infty)=H^2(\R^2),\quad 
    -\Delta_\infty f
    =-\Delta f. \]
Here $\lambda>0$ is a fixed constant with $\lambda\ne|e_\alpha|$ (see \eqref{eq:negaengen} for the definition of $e_\al$),
\begin{equation*}\label{eq.2d1}
   \beta(\lambda) :=\frac{1}{2\pi}\Bigl(\gamma + \log \frac{\sqrt{\lam}}{2}\Bigr),
\end{equation*}
$\gamma>0$ is the Euler--Mascheroni constant, and $G_{\lam}$ is the Green function of $-\Delta+\lam$ on $\R^2$, which satisfies the distributional relation $(-\Delta+\lam)G_{\lam}=\delta_0$ and is given by
\begin{equation}\label{eq:expG1}
  G_{\lam}(x) = \frac1{2\pi} \mathcal{F}^{-1} \Bigl[\frac{1}{|\xi|^2+\lam}\Bigr](x) = \frac1{(2\pi)^2} \int_{\R^2} \frac{e^{ix\cdot\xi}}{|\xi|^2+\lam} \, d\xi,
\end{equation}
where $\mathcal{F}^{-1}$ is the inverse Fourier transform. Note that $f(0)$ makes sense for $f\in H^2(\R^2)$ by the embedding $H^2(\R^2)\hookrightarrow C(\R^2)$. From the expression~\eqref{eq:expG1}, we see that
\begin{align*}
\label{eq:GlamH1ep}
    &G_\lam \in H^{s}(\R^2),\quad s<1,&
    &G_\lam \notin H^1(\R^2).
\end{align*}
The function $G_\lam$ is also represented as
\begin{equation}\label{eq:defGlambda}
  G_\lam(x) = \frac{1}{2\pi} K_{0}(\sqrt{\lam}|x|),
\end{equation}
with the modified Bessel function $K_0$ of the second kind (or the Macdonald function) of order zero. From the expression~\eqref{eq:defGlambda} and \cite[Chapter~10]{NIST:DLMF}, we see that $G_\lam$ is positive, radial, and exponentially decreasing function with the asymptotics
\begin{equation}\label{eq:Gexpdecay}
    G_\lam(r) = -\frac{1}{2\pi}\log r-\beta(\lam) +o(1)\quad \text{as } r \to 0.
\end{equation}
Note that $-(2\pi)^{-1}\log r$ corresponds to the fundamental solution of the Laplacian in $\R^2$. We can easily check that the definition \eqref{eq:defpi} of $-\Delta_\alpha$ is independent of the choice of $\lambda$. 

The spectral properties of $-\Delta_\alpha$ have already been known (see \cite{AlbeGeszHo-eHold05} for more details). For $\alpha\in\R$, the essential spectrum is given by
\[  \sigma_{\mathrm{ess}}(-\Delta_\alpha)
    =\sigma(-\Delta)
    =[0, \infty). \]
Moreover, $-\Delta_\alpha$ has exactly one simple negative eigenvalue
\begin{equation} \label{eq:negaengen}
    e_\alpha:=-4e^{-4\pi\alpha-2\gamma}<0
\end{equation}
and the corresponding eigenfunction $G_{|e_\alpha|}$. The spectrum of $-\Delta_\alpha$ is given by 
\[  \sigma(-\Delta_\alpha)
    =\sigma_{\mathrm{p}}(-\Delta_\alpha)\cup\sigma_{\mathrm{ess}}(-\Delta_\alpha)
    =\{e_\alpha\}\cup[0, \infty).   \]
The form domain $H_\alpha^1(\R^2):=D[-\Delta_\alpha]$ is given by 
\[  H_\alpha^1(\R^2)
    =\left\{\begin{alignedat}{2}
   &\set{f+cG_\lambda}{f\in H^1(\R^2),\: c\in\mathbb{C}} &\quad
   &\text{if }\alpha\in\R,
\\ & \omit \hfill $H^1(\R^2)$\hfill  &\quad 
   &\text{if }\alpha=\infty,
    \end{alignedat}\right. \]
and the quadratic form on $H_\alpha^1(\R^2)$ is defined by 
\begin{equation}\label{eq:norm}
    \langle (-\Delta_\alpha+\lambda)v, v\rangle 
    =\|\nabla f\|_{L^2}^2
    +\lambda\|f\|_{L^2}^2
    +(\alpha+\beta(\lambda))|c|^2    
\end{equation}  
for $v=f+cG_\lambda\in H_\alpha^1(\R^2)$. If $\lambda>|e_\alpha|$, the square root of \eqref{eq:norm} is a norm in $H_\alpha^1(\R^2)$.

We define the action for \eqref{eq:sp} by 
\[  S_\lambda(v)
    :=\frac12\langle (-\Delta_\alpha+\lambda)v, v\rangle 
    -\frac1{p+1}\|v\|_{L^{p+1}}^{p+1},\quad 
    v\in H_\alpha^1(\R^2).  \]
Then $u\in H_\alpha^1(\R^2)$ solves \eqref{eq:sp} if and only if $S_\lambda'(u)=0$. We define the set of nontrivial solutions of \eqref{eq:sp} by
\[  \mathcal{A}_\lambda
    :=\set{u\in H_\alpha^1(\R^2)}
    {u\ne 0,\: S_\lambda'(u)=0} \]
and the set of \emph{ground states} by
\begin{equation*} \label{eq:GS}
    \mathcal{G}_\lambda 
    :=\set{u\in\mathcal{A}_\lambda}
    {S_\lambda(u)\le S_\lambda(v),\:\forall v\in \mathcal{A}_\lambda}. 
\end{equation*}
From the embedding $H_\alpha^1(\R^2)\hookrightarrow L^q(\R^2)$ for $q\ge2$, we see from the equation \eqref{eq:sp} that $\mathcal{A}_\lambda\subset D(-\Delta_\alpha)$. 

The following results on existence and symmetry have been already known.

\begin{proposition}[\cite{AdamBoniCarlTent22-2d, FukaGeorIked22}] \label{prop:1}
Let $p>1$ and $\alpha\in\R$. The following holds.
\begin{itemize}
\item If $\lambda>|e_\alpha|$, then $\mathcal{G}_\lambda\ne\emptyset$.
\item If $\lambda>|e_\alpha|$, $u\in \mathcal{G}_\lambda$, and $u$ is decomposed as $u=f+f(0)(\alpha+\beta(\lambda))^{-1}G_\lambda$, then there exists $\theta\in\R$ such that $e^{i\theta}f$ is a positive, radial, and decreasing function. In particular, so is $e^{i\theta}u$.
\end{itemize}
\end{proposition}

First we state our main results on uniqueness.

\begin{theorem} \label{thm:1}
For $p>1$, $\alpha\in\R$, and $\lambda>|e_\alpha|$, the equation \eqref{eq:sp} has at most one positive radial solution. 
\end{theorem}

\begin{remark}
We can easily check by taking the $L^2$-inner product of \eqref{eq:sp} and the eigenfunction $G_{|e_\alpha|}$ of $-\Delta_\alpha$ that if $\lambda\le |e_\alpha|$, then \eqref{eq:sp} does not have any positive solutions. 
\end{remark}

As a consequence of Theorem~\ref{thm:1} combined with Proposition~\ref{prop:1}, we have the uniqueness of ground states.

\begin{theorem} 
For $p>1$, $\alpha\in\R$, and $\lambda>|e_\alpha|$, there exists the unique positive, radial, and decreasing solution $u_\lambda$ of \eqref{eq:sp} such that 
\[  \mathcal{G}_\lambda 
    =\set{e^{i\theta}u_\lambda}
    {\theta\in \R}. \]
\end{theorem}

Next, we state our results on nondegeneracy. The linearized operator $S_{\lambda}''(u_\lambda)$ around the unique positive ground state $u_\lambda\in\mathcal{G}_\lambda$ has the expression 
\[  S_{\lambda}''(u_\lambda) v
    =(-\Delta_\alpha+\lambda)v 
    -pu_\lambda^{p-1}\Re v
    -i u_\lambda^{p-1}\Im v
    \]
for $v\in H_\alpha^1(\R^d)$. This can be decomposed into real part and imaginary part as 
\[  \langle S_{\lambda}''(u_\lambda) v, w\rangle 
    =\langle L_\lambda^+ \Re v, \Re w\rangle
    +\langle L_\lambda^- \Im v, \Im w\rangle, \]
where 
\begin{align} \label{eq:Lpm}
    L_\lambda^+
    &:=-\Delta_\alpha+\lambda 
    -pu_\lambda^{p-1},&
    L_\lambda^-
    &:=-\Delta_\alpha+\lambda 
    -u_\lambda^{p-1}.  
\end{align}

\begin{theorem} \label{thm:2}
For $p>1$, $\alpha\in\R$, and $\lambda>|e_\alpha|$, the unique positive ground state $u_\lambda\in \mathcal{G}_\lambda$ is a nondegenerate critical point of $S_{\lambda}|_{H_\alpha^1(\R^2; \R)}$. More precisely, if $w\in H_\alpha^1(\R^2; \R)$ satisfies $L_\lambda^+w=0$, then $w= 0$. 
\end{theorem}

\begin{remark}
Note that $L_{\lambda}^-u_{\lambda}=0$ from \eqref{eq:sp}. Since $u_{\lambda}$ is positive, we see that $u_{\lambda}$ is the ground state of $L_{\lambda}^-$ and $\ker L_{\lambda}^-=\operatorname{span}\{u_\lambda\}$.
\end{remark}

The uniqueness and nondegeneracy of positive radial solutions for the nonlinear scalar field equation 
\begin{equation*}\label{eq:spwopd}
    (-\Delta+\lambda)u 
    -u^p
    =0,\quad 
    u\in H^1(\R^d)
\end{equation*} 
have been well studied, where $d\in\mathbb{N}$, $p>1$, and $\lambda>0$. For $d=3$ and $p=3$, the uniqueness was firstly proved by \cite{Coff72}. This was extended by \cite{McLeSerr87} and completely solved by \cite{Kwon89}. The nondegeneracy is obtained in \cite{Kwon89, Wein85} by using the uniqueness results. Our results (Theorems~\ref{thm:1} and \ref{thm:2}) are analogies of these results for the case with a point interaction.

To prove Theorems~\ref{thm:1} and \ref{thm:2} we mainly follow the argument of \cite{ShioWata13, ShioWata16}. First, we show that any positive radial solution $u$ of \eqref{eq:sp} satisfies \eqref{eq:ODE}. Note that the boundary condition at the origin is natural because the two functions 
\begin{align*}
    u_1(r)
   &=-\frac{1}{2\pi}\log r,&
    u_2(r)
   &=1 \text{ (constant)}
\end{align*}
compose a system of fundamental solutions of the Euler ordinary differential equation $-u''-u'/r=0$, $r>0$ associated with the Laplace equation in $\R^2$. We also note that the case $\alpha=\infty$ ($-\Delta_\alpha=-\Delta$) corresponds to the boundary condition $u(0)=y_u>0$. Next, we establish the uniqueness of positive solutions of \eqref{eq:ODE} by using the Poho\v{z}aev identities derived in \cite{ShioWata13} (see \cite{Yana91} for the original work by a method of Poho\v{z}aev identities). The proof of the nondegeneracy result (Theorem~\ref{thm:2}) is divide into two steps. We firstly treat the case in the radial function space $H_{\alpha, \rad}^1(\R^2)$ based on \cite{KabeTana99}. We characterize the ground state $u_\lambda$ of \eqref{eq:sp} as that of perturbed equations \eqref{eq:elldel}  below. Then using the uniqueness results and the variational characterization of the ground state, we establish the nondegeneracy in the radial function space $H_{\alpha, \rad}^1(\R^2)$. We secondly treat the case in the whole function space $H_{\alpha}^1(\R^2)$ based on \cite{NiTaka93}. We expand the function $w\in \ker L_\lambda^+$ as a linear combination $w=\sum_{j=0}^\infty w_j(r)e_j(\omega)$ of spherical harmonics $(e_j)_{j=0}^\infty$ and show $w_j= 0$ for all $j\ge 0$ by using an ODE argument. 

In our proof, however, some adaptations of the arguments to our problem are required as follows:
\begin{itemize} 
\item We can use the same Poho\v{z}aev identities as in the case $\alpha=\infty$, since the differential equations are identical. However, more suitable analysis near the origin is necessary to establish uniqueness due to the singular boundary condition. Similar modifications were done in \cite{Fuka21} for the nonlinear scalar field equations with singular external potentials. 

\item To apply the arguments of \cite{KabeTana99}, we need to show the existence of positive ground states for the perturbed equation \eqref{eq:elldel}. For this purpose, the following claim is useful for the unperturbed equation \eqref{eq:sp} (see \cite{FukaGeorIked22, GeorMichScan24}):
\begin{equation} \label{eq:1.15} 
f+cG_\lambda \in H_\alpha^1(\R^2) \text{ is a ground state} \implies \text{so is }|f|+|c|G_\lambda. 
\end{equation} 
However, this does not hold for \eqref{eq:elldel} due to the perturbed potential term, which depends on the small parameter $\varepsilon > 0$. To address this issue, we construct the fundamental solution $G_{\lambda, \varepsilon}$ of the perturbed Helmholtz equation (Lemma~\ref{lem:5.4}) and use the modified claim \eqref{eq:1.15}, where $G_\lambda$ is replaced by $G_{\lambda, \varepsilon}$ (Proof of Lemma~\ref{lem:5.7}). See also \cite{BoniGall25} for the similar argument in the case involving the Coulomb potential and a point interaction.

\item In the case without point interaction $(\alpha = \infty)$, 
$L_\lambda^+$ has additional elements $\partial_j u_\lambda$ composing its kernel in the whole function space $H^1(\R^2)$ due to translation invariance. In the case with point interaction $(\alpha \in \R)$, such possibilities are excluded in view of the singularity at the origin (see Case 2 in the Proof of Theorem~\ref{thm:2}). 
\end{itemize}

We remark that the same problem can be considered for nonlinear scalar field equations with a point interaction in the three dimensional case with $1<p<2$ (see \cite{AdamBoniCarlTent22-3d} for the existence and symmetry for the ground states). The corresponding boundary value problem of the ordinary differential equation is 
\begin{equation} \label{eq:ODE3d}
    \left\{\begin{alignedat}{2}
   &\mathopen{}
    - u''
    -\frac{2}{r} u' 
    +\lambda u 
    -u^p 
    =0,& \quad
   &r>0,
\\ &\mathopen{}
    u(r)
    =y_u\Bigl(\frac{1}{4\pi|x|}
    +\alpha\Bigr)+o(1)&\quad
   &\text{as }r\downarrow 0
    \text{ for some }y_u>0,
\\ &\mathopen{}u(r)\to 0&\quad
   &\text{as }r\to\infty.
    \end{alignedat}\right.
\end{equation}
However, our argument does not work to prove the uniqueness of positive solution for \eqref{eq:ODE3d}. Indeed, the positivity \eqref{eq:Jposi} of the Poho\v{z}aev function $J[u](r)$ is crucial, but we can show $J[u](r)\to-\infty$ as $r\downarrow 0$ in the case of \eqref{eq:ODE3d}. See Remark~\ref{rem:3d} for more details.

The rest of this paper is organized as follows. In Section~\ref{sec:2}, we state several properties of the solutions of \eqref{eq:sp} and \eqref{eq:ODE}. In Section~\ref{sec:3}, we prove Theorem~\ref{thm:1} by using the Poho\v{z}aev identities. In Section~\ref{sec:NDrad}, we prove the nondegeneracy in the radial function space $H_{\alpha, \rad}^1(\R^2)$. In Section~\ref{sec:NDwhole}, we prove the nondegeneracy in the whole function space $H_\alpha^1(\R^2)$ and completes the proof of Theorem~\ref{thm:2}. Hereafter, for a function space $X$, we put the subscript `$\mathrm{r}$' as $X_\mathrm{r}$ to denote the subspace of $X$ consisting of radially symmetric functions. Furthermore, we identify a radially symmetric function on $\R^2$ with a function on $[0, \infty)$ as $u(x)=u(|x|)=u(r)$ without further notice.

\section{Preliminaries}\label{sec:2}

In this section, we prepare several lemmas about the solutions of \eqref{eq:sp}. In what follows, we assume $p>1$, $\alpha\in\R$, and $\lambda>|e_\alpha|$.

\begin{lemma} \label{lem:2.1}
Let $u$ be a positive radial solution of \eqref{eq:sp}. Then $u(r)$ belongs to $C^2(0, \infty)$ and is a positive solution of \eqref{eq:ODE}. 
\end{lemma}

\begin{proof}
By a elliptic regularity argument, e.g., \cite[Theorem~11.7]{LiebLoss01}, we see that 
$
    \mathcal{A}_\lambda\subset C^2(\R^2\setminus\{0\}).
$
Thus, $u\in C^2(0, \infty)$. 

Since $\mathcal{A}_\lambda\subset D(-\Delta_\alpha)$, we can decompose $u=f+f(0)(\alpha+\beta(\lambda))^{-1}G_\lambda$ with $f\in H^2(\R^2)$. The action of $-\Delta_\alpha+\lambda$ can be rewritten as 
\begin{align*}
    (-\Delta_\alpha+\lambda)u
    =(-\Delta+\lambda)f
   &=(-\Delta+\lambda)\Bigl(u-\frac{f(0)}{\alpha+\beta(\lambda)}G_\lambda\Bigr)
\\ &=(-\Delta+\lambda)u-\frac{f(0)}{\alpha
    +\beta(\lambda)}\delta_0. 
\end{align*}
In particular, 
\[  (-\Delta_\alpha+\lambda)u
    =(-\Delta+\lambda)u\quad 
    \text{for }x\ne 0.  \]
This implies that $u$ satisfies $-u''-u'/r+\lambda u-u^p=0$, $r>0$. 

Moreover, by \eqref{eq:Gexpdecay}, we can rewrite the boundary condition for $u$ at $x=0$  as
\begin{align*}
\notag 
    u(x)
   &=f(x)+\frac{f(0)}{\alpha+\beta(\lambda)}G_\lambda(x)
\\ &=\frac{f(0)}{\alpha+\beta(\lambda)}\Bigl(-\frac{1}{2\pi}\log |x|+\alpha\Bigr)
    +o(1)\quad 
    \text{as }x\to 0.
\end{align*}
We note that $\alpha+\beta(\lambda)>0$ because of $\lambda>|e_\alpha|$. Therefore, if we put $y_u:=f(0)(\alpha+\beta(\lambda))^{-1}>0$, then $u(r)=y_u(-(2\pi)^{-1}\log r+\alpha)+o(1)$ as $r\downarrow 0$.

Finally, since $u\in D(-\Delta_\alpha)$, we have $u(r)\to 0$ as $r\to\infty$. This completes the proof.
\end{proof}

In what follows in this section, we investigate the properties of the solution of 
\begin{alignat}{2} \label{eq:ODE1}
   &\mathopen{}
    - u''
    -\frac{1}{r} u' 
    +\lambda u 
    -u^p 
    =0,& \quad
   &r>0,
\\  \label{eq:ODE2}&\mathopen{}
    u(r)
    =y_u\Bigl(-\frac{1}{2\pi}\log r
    +\alpha\Bigr)+o(1)&\quad
   &\text{as }r\downarrow 0
    \text{ for some }y_u>0,
\\  \label{eq:ODE3}
   &\mathopen{}u(r)\to 0&\quad
   &\text{as }r\to\infty.
\end{alignat}

\begin{lemma}
Let $u\in C^2(0, \infty)$ be a positive solution of \eqref{eq:ODE1} satisfying \eqref{eq:ODE2}. Then
\begin{align}\label{eq:2.1}
    u'(r) 
   &=-\frac{y_u}{2\pi r}
    +\frac{1}{r}\int_0^r s\bigl(\lambda u(s)- u(s)^p\bigr)\,ds,
\\ \label{eq:n2.2} 
    u(r)
   &=y_u\Bigl(-\frac{1}{2\pi}\log r+\alpha\Bigr)
    +\int_0^r\int_{0}^t\frac{s}{t}\bigl(\lambda u(s)- u(s)^p\bigr)\,ds\,dt
\end{align}
for all $r>0$. In particular, 
\begin{equation}\label{eq:n2.3} 
    u'(r)
    =-\frac{y_u}{2\pi r}+o(1)
\end{equation}
as $r\downarrow 0$.
\end{lemma}

\begin{proof}
We rewrite \eqref{eq:ODE1} as 
\[  \frac{d}{ds}(su'(s))
    =s(\lambda u(s)-u(s)^p),\quad 
    s>0.    \]
Integrating this on $[\varepsilon, t]$, we have 
\begin{align*}
    tu'(t)
    -\varepsilon u'(\varepsilon)
    =\int_{\varepsilon}^t s\bigl(\lambda u(s)-u(s)^p\bigr)\,ds. 
\end{align*}
Since $s(\log s)^q\in L^1(0, 1)$ for any $q\ge 0$, the limit of the right-hand side as $\varepsilon\downarrow 0$ exists, and so does $l:=\lim_{\varepsilon\downarrow 0}\varepsilon u'(\varepsilon)$. Thus, we have
\begin{align}\label{eq:2.2}
    tu'(t)
    -l
    =\int_{0}^t s\bigl(\lambda u(s)-u(s)^p\bigr)\,ds. 
\end{align}
Moreover, multiplying this by $1/t$ and integrating it over $[\varepsilon, r]$, we obtain
\begin{align} \label{eq:2.3}
     u(r)
    - u(\varepsilon)
    -l(\log r-\log \varepsilon)
   &=\int_\varepsilon^r\int_{0}^t\frac{s}{t}\bigl(\lambda u(s)-u(s)^p\bigr)\,ds\,dt.
\end{align}
If we take $r=1$, this equality is rewritten as 
\[  u(1)
    - u(\varepsilon)
    +l\log\varepsilon
    =\int_\varepsilon^1\int_{0}^t\frac{s}{t}\bigl(\lambda u(s)-u(s)^p\bigr)\,ds\,dt  \]
and its right hand side is estimated as
\begin{align*}
    \Bigl|\int_\varepsilon^1\int_{0}^t\frac{s}{t}\bigl(\lambda u(s)- u(s)^p\bigr)\,ds\,dt\Bigr|
    \le \int_0^1\lambda u(s)+ u(s)^p\,ds<\infty.
\end{align*}
Since $u(\varepsilon)=-y_u(2\pi)^{-1}\log \varepsilon +O(1)$ as $\varepsilon\downarrow 0$, we have $l=-y_u(2\pi)^{-1}$. Thus, \eqref{eq:2.1} follows from \eqref{eq:2.2}, and \eqref{eq:n2.2} follows from \eqref{eq:2.3} and $u(\varepsilon)=-y_u(2\pi)^{-1}(\log \varepsilon +\alpha) +o(1)$ as $\varepsilon\downarrow 0$.
\end{proof}

\begin{lemma} \label{eq:GronUniq}
Let $ u, v\in C^2(0, \infty)$ be two positive solutions of \eqref{eq:ODE1} satisfying \eqref{eq:ODE2}. If $y_u=y_v$, then $ u= v$. 
\end{lemma}

\begin{proof}
Let $R>0$. By \eqref{eq:n2.2} and $y_u=y_v$ we obtain for $0<r\le R$ the estimate 
\begin{align*}
   &|u(r)- v(r)|
   \le\int_0^r\int_{0}^t\frac{s}{t}\big(\lambda |u(s)- v(s)|+|u(s)^p- v(s)^p|\big)\,ds\,dt
\\ &=\int_0^r\int_s^r\frac{s}{t}\big(\lambda|u(s)- v(s)|
   +|u(\tau)^p- v(\tau)^p|\big)\,dt\,ds
\\ &\lesssim \int_0^r\Big(\int_s^{R}\frac{dt}{t}\Big)s(1+u(s)^{p-1}+v(s)^{p-1})|u(s)-v(s)|\,ds.
\end{align*}
We note that $h(s):=(\int_s^{R}\frac{dt}{t})s(1+ u(s)^{p-1}+u(s)^{p-1})$ and $w(s):=|u(s)-v(s)|$ are extended as continuous functions on $[0, R]$ with $w(0)=0$. Thus, Gronwall's inequality implies $u= v$ on $(0, R]$. Since $R>0$ is taken to be arbitrary, we have $u= v$ on $(0, \infty)$.
\end{proof}

\begin{lemma} \label{lem:w'}
Let $ u, v\in C^2(0, \infty)$ be two positive solutions of \eqref{eq:ODE1} satisfying \eqref{eq:ODE2}. Then
\begin{equation}\label{eq:expw}
    \frac{d}{dr}\Bigl(\frac{v(r)}{u(r)}\Bigr)
    =-\frac{1}{ru(r)^2}\int_0^r s\Bigl(\Bigl(\frac{v(s)}{u(s)}\Bigr)^{p-1}-1\Bigr)u(s)^p v(s)\,ds
\end{equation}
for all $r>0$. In particular, 
\begin{equation} \label{eq:wo}
    \frac{d}{dr}\Bigl(\frac{v(r)}{u(r)}\Bigr)=o(1)\quad 
    \text{as }
    r\downarrow 0. 
\end{equation}
\end{lemma}

\begin{proof}
By \eqref{eq:ODE1}, We have
\begin{align}
   &\notag 
    \frac{d}{ds}\Bigl(s\bigl(v'(s)u(s)-v(s)u'(s)\bigr)\Bigr)
\\ &\notag
    =v'(s)u(s)-v(s)u'(s)+s\bigl(v''(s)u(s)-v(s)u''(s)\bigr)
\\ &=s\big(u(s)^{p-1}-v(s)^{p-1}\big)u(s)v(s)
    =-s\Big(\Bigl(\frac{v(s)}{u(s)}\Bigr)^{p-1}-1\Big)u(s)^pv(s)\label{eq:orv}
\end{align}
Then by \eqref{eq:ODE2} and \eqref{eq:n2.3} we have 
\begin{align*}
    v'(\varepsilon)u(\varepsilon)-v(\varepsilon)u'(\varepsilon)
   &=\begin{multlined}[t]
    \Bigl(-\frac{y_v}{2\pi \varepsilon}+o(1)\Bigr)\cdot y_u\Bigl(-\frac{1}{2\pi}\log \varepsilon+\alpha+o(1)\Bigr)
\\  -y_v\Bigl(-\frac{1}{2\pi}\log \varepsilon+\alpha+o(1)\Bigr)\cdot\Bigl(-\frac{y_v}{2\pi \varepsilon}+o(1)\Bigr)
    \end{multlined}
\\ &=o(\varepsilon^{-1})\quad 
    \text{as }\varepsilon\downarrow 0.
\end{align*}  
Therefore, integrating \eqref{eq:orv} on $[\varepsilon, r]$ and taking the limit $\varepsilon\downarrow 0$, we obtain 
\[  r\bigl(v'(r)u(r)-v(r)u'(r)\bigr)
    =-\int_0^r s\Bigl(\Bigl(\frac{v(s)}{u(s)}\Bigr)^{p-1}-1\Bigr)u(s)^pv(s)\,ds.   \]
From this, we have \eqref{eq:expw}. The estimate \eqref{eq:wo} follows from \eqref{eq:expw}.
\end{proof}

\begin{lemma} \label{lem:uu'exp}
Let $u\in C^2(0, \infty)$ be a positive solution of \eqref{eq:ODE1} satisfying \eqref{eq:ODE3}. Then there exist $\varepsilon>0$ and $C>0$ such that 
\[  |u(r)|+|u'(r)|
    \le Ce^{-\varepsilon r},\quad 
    r>1.  \]
\end{lemma}

\begin{proof}
See \cite[Lemma~2]{BereLion83}.
\end{proof}

\section{Uniqueness of ground states} \label{sec:3}

In this section we prove Theorem~\ref{thm:1}. Firstly, we review the Poho\v{z}aev identities by \cite{ShioWata13}. We consider the more general ordinary differential equation:
\begin{equation} \label{eq:gODE}
    -u''
    -\frac{d-1}{r}u' 
    +g(r)u
    -h(r)u^p
    =0,\quad 
    r>0,
\end{equation}
where $d>1$, $g, h\in C^1(0, \infty)$, and $h\ge 0$.
For $u\in C^2(0, \infty)$ we define
\begin{align*}
    J[u](r)
    :=\begin{multlined}[t]
    \frac12a(r)u'(r)^2
    +b(r)u'(r)u(r) 
\\  +\frac12\bigl(c(r)-a(r)g(r)\bigr)u(r)^2
    +\frac{1}{p+1}a(r)h(r)u(r)^{p+1},\quad 
    r>0,
    \end{multlined}
\end{align*}
where $a(r)$, $b(r)$, and $c(r)$ will be defined later. If $u$ is a solution of \eqref{eq:gODE}, we differentiate $J[u]$ to have 
\begin{align*}
    J[u]'
   &=\begin{multlined}[t]
    \frac12a'(u')^2
    +au''u'
    +b'u'u
    +bu''u
    +b(u')^2
\\  +\frac12(c-ag)'u^2
    +(c-ag)u'u
    +\frac{1}{p+1}(ah)'u^{p+1}
    +ahu'u^{p}
    \end{multlined}
\\ &=\begin{multlined}[t]
    \frac12a'(u')^2
    +a\Bigl(-\frac{d-1}{r}u'+gu-hu^p\Bigr)u'
    +b'u'u
    +b\Bigl(-\frac{d-1}{r}u'+gu-hu^p\Bigr)u
\\  +b(u')^2
    +\frac12(c-ga)'u^2
    +(c-g a)u'u
    +\frac{1}{p+1}(ah)'u^{p+1}
    +ahu'u^{p}
    \end{multlined}
\\ &=A(r)(u')^2
    +B(r)u'u
    +C(r)u^2
    +D(r)u^{p+1},
\end{align*}
where 
\begin{align*}
    A(r)
    &:=\frac{1}{2}a'(r)-\frac{d-1}{r}a(r)+b(r),&
    B(r)
    &:=b'(r)-\frac{d-1}{r}b(r)+c(r),
\\  C(r)
    &:=b(r)g(r)
    +\frac12(c-ag)'(r),&
    D(r)
    &:=-b(r)h(r)+\frac{1}{p+1}(ah)'(r).
\end{align*}
We solve the system of differential equations $A=B=D=0$. From $A=0$ and $D=0$ we have
\[  \frac12a'(r) 
    -\frac{d-1}{r}a(r)
    +\frac{1}{p+1}\frac{(ah)'(r)}{h(r)}
    =0, \] 
which can be rewritten as 
\[  \bigl(\log \lvert a(r)\rvert\bigr)'
    =\Bigl(\log\bigl((r^{d-1})^\frac{2(p+1)}{p+3}h(r)^{-\frac{2}{p+3}}\bigr)\Bigr)'.  \]
Let $a(r)$ be its one solution:
\begin{equation}\label{eq:a(r)}
    a(r)
    :=(r^{d-1})^\frac{2(p+1)}{p+3}h(r)^{-\frac{2}{p+3}},\quad 
    r>0.
\end{equation}
Moreover, let $b(r)$ be the function determined by $A(r)\equiv 0$ as
\begin{align}
    \label{eq:b(r)}
    b(r) 
   &:=\frac{d-1}{r}a(r)-\frac12a'(r),
\end{align} 
and $c(r)$ be the function determined by $B(r)\equiv 0$ as
\begin{align}\label{eq:c(r)}
    c(r)
   &:=\frac{d-1}{r}b(r)-b'(r).
\end{align}
Moreover, we put
\begin{align}\label{eq:C(r)} 
    C(r) 
   &:=b(r)g(r)
    +\frac12(c-ag)'(r).
\end{align}
Then we have the following.

\begin{lemma}
Let $u\in C^2(0, \infty)$ be a positive solution of \eqref{eq:gODE} and let $a(r)$, $b(r)$, $c(r)$, and $C(r)$ be given by \eqref{eq:a(r)}--\eqref{eq:C(r)}. Then $J[u]'(r)=C(r)u(r)^2$ for all $r>0$. 
\end{lemma}

We now turn to our problem \eqref{eq:ODE1}. This corresponds to \eqref{eq:gODE} with $d=2$, $g(r)=\lambda$, and $h(r)=1$. In this case we have the following expression:
\begin{align*}
    a(r)
   &=r^q,\quad 
    q:=\frac{2(p+1)}{p+3}\in(1, 2),
\\  b(r)    
   &=\Bigl(1-\frac{q}{2}\Bigr)r^{q-1},
\\  c(r)
   &=\frac12(2-q)^2r^{q-2},
\\  C(r)
   &=-\frac12(2-q)^3r^{q-3}
   -(q-1)\lambda r^{q-1}.
\end{align*}
We note that $C(r)<0$ for all $r>0$, which is convenient in our proof.

\begin{lemma}
If $u\in C^2(0, \infty)$ is a positive solution of \eqref{eq:ODE1} satisfying \eqref{eq:ODE3}, then the following holds: 
\begin{align}
   &\label{eq:Jto0} 
   \lim_{r\to\infty}J[u](r)=0,
\\ \label{eq:Jposi} 
   &J[u](r)>0,\quad r>0.
\end{align}
\end{lemma}

\begin{proof}
\eqref{eq:Jto0} follows from Lemma~\ref{lem:uu'exp}. \eqref{eq:Jposi}  follows from \eqref{eq:Jto0}, $J[u]'(r)=C(r)u(r)^2$, and $C(r)<0$ for all $r>0$. 
\end{proof}

\begin{remark} \label{rem:3d}
In the three-dimensional case of \eqref{eq:ODE3d}, which corresponds to \eqref{eq:gODE} with \( d=3 \), \( 1<p<2 \), \( g(r) = \lambda \), and \( h(r) = 1 \), the function \( C(r) \) defined in \eqref{eq:C(r)} can be explicitly written as:
\[  C(r)
    =\frac12(3-q)(q-2)\Bigl(2-\frac{q}{2}\Bigr)r^{q-3}
    -\lambda(q-2)\lambda^{q-1},\quad 
    q:=\frac{4(p+1)}{p+3}\in(2, 3).  \]
From this expression, it follows that there exists \( r_0 > 0 \) such that \( C(r) > 0 \) for \( r \in (0, r_0) \) and \( C(r) < 0 \) for \( r \in (r_0, \infty) \). This behavior of \( C(r) \) is useful to show the positivity of \( J[u] \) in the case without point interaction (\( \alpha = \infty \)). However, for the case with point interaction (\( \alpha \in \mathbb{R} \)), this shape of \( C(r) \) does not suffice. Furthermore, a direct calculation shows that \( J[u](r) \to -\infty \) as \( r \downarrow 0 \), due to the singularity of \( u(r) \). Consequently, our argument for \( J[u] \) does not extend to the three dimensional case \eqref{eq:ODE3d}.
\end{remark}

Let $u(r)$ and $v(r)$ be two positive solution of \eqref{eq:ODE}. We put 
\[  w(r) 
    :=\frac{v(r)}{u(r)} \]
and 
\begin{align*}
    X(r) 
    :={}&w(r)^2J[u](r) 
    -J[v](r)
\end{align*}
for $r>0$. Note that this can be expressed as 
\begin{align}
    \notag 
    X&=\frac{a}{2}\Bigl(\frac{v^2(u')^2}{u^2}-(v')^2\Bigr)
    +b\Bigl(\frac{v^2u'}{u}-v'v\Bigr)
    +\frac{a}{p+1}v^2(u^{p-1}-v^{p-1})
\\ &\label{eq:xformula}
    =-\frac{a}{2}(u'v+uv')w'-buvw'
    +\frac{a}{p+1}u^{p-1}v^2(1-w^{p-1}).
\end{align}

\begin{lemma} \label{lem:Xto0} 
$\lim_{r\downarrow 0}X(r)=0$.
\end{lemma}

\begin{proof}
By the expression~\eqref{eq:xformula} and by the estimates \eqref{eq:ODE2}, \eqref{eq:n2.3}, \eqref{eq:wo}, and $a(r)=O(r^q)$, $b(r)=O(r^{q-1})$ with $q>1$, we have the assertion.
\end{proof}

By the Poho\v{z}aev identity we have 
\begin{equation*} \label{eq:X'}
    X'(r) 
    =2w'(r)w(r)J[u](r),\quad 
    r>0.    
\end{equation*}

\begin{lemma} \label{lem:w'X'nega}
If $y_v>y_u$, then $w'(r)<0$ for all $r>0$. In particular, $X'(r)<0$ for all $r>0$.
\end{lemma}

\begin{proof}
See \cite[Proposition 3]{ShioWata16}.
\end{proof}


\begin{proposition} \label{prop:uniq}
The problem \eqref{eq:ODE} has at most one positive solution.
\end{proposition}

\begin{proof}
We assume that there is two positive solutions $u$ and $v$ of \eqref{eq:ODE}. Suppose that $y_u<y_v$. By Lemmas~\ref{lem:Xto0} and \ref{lem:w'X'nega}, we have $\limsup_{r\to\infty}X(r)\in[-\infty, 0)$. On the other hand, noting that $w$ is bounded by Lemma~\ref{lem:w'X'nega}, we obtain $\lim_{r\to\infty}X(r)=0$ from \eqref{eq:Jto0}. This is a contradiction. Thus, $y_u=y_v$. Then Lemma~\ref{eq:GronUniq} implies $u= v$. This completes the proof.
\end{proof}

\begin{proof}[Proof of the Theorem~\ref{thm:1}]
The assertion follows from Lemma~\ref{lem:2.1} and Proposition~\ref{prop:uniq}. 
\end{proof}

\section{Nondegeneracy in the radial function space}\label{sec:NDrad}

Let $u_\lambda\in\mathcal{G}_\lambda$ be the unique positive ground state for \eqref{eq:sp} and let $L_\lambda^+$ be defined in \eqref{eq:Lpm}. In this section we prove the following nondegeneracy in the radial function space, which is a special case of Theorem~\ref{thm:2}. 

\begin{theorem} \label{prop:nondrad}
For $p>1$, $\alpha\in\R$, and $\lambda>|e_\alpha|$, the unique positive ground state $u_\lambda\in \mathcal{G}_\lambda$ is a nondegenerate critical point of $S_{\lambda}|_{H_{\alpha, \rad}^1(\R^2; \R)}$. More precisely, if $w\in H_{\alpha, \rad}^1(\R^2; \R)$ satisfies $L_\lambda^+w=0$, then $w= 0$.
\end{theorem}

Let $\chi\in C^\infty_c(0, \infty)$ satisfy $0\le \chi\le 1$, $\chi=0$ on $(0, 1]\cup[3, \infty)$, and $\chi(2)=1$. For $\varepsilon>0$ we define 
\begin{align*}
    g_{\lambda, \varepsilon}(r)
   &:=\lambda +\varepsilon\chi(r)u_\lambda(r)^{p-1},&
    h_\varepsilon(r)
   &:=1+\varepsilon\chi(r)
\end{align*}
for $r>0$. Then $u_\lambda$ solves 
\begin{equation} \label{eq:elldel}
    \bigl(-\Delta_\alpha+g_{\lambda, \varepsilon}(|x|)\bigr)u 
    -h_\varepsilon(|x|)|u|^{p-1}u
    =0,\quad u\in H_{\alpha, \rad}^1(\R^2).
\end{equation}

We put 
\begin{align*}
    Q_{\lambda, \varepsilon}(v)
   &:=\langle (-\Delta_\alpha+g_{\lambda, \varepsilon}(|x|))v, v\rangle,&
    N_{\varepsilon}(v)
   &:=\int_{\R^2}h_\varepsilon(|x|)|v|^{p+1}\,dx.&
\end{align*}
Note that $Q_{\lambda, \varepsilon}(v)^{1/2}$ and $N_\varepsilon(v)^{1/(p+1)}$ are equivalent to the $H_\alpha^1(\R^2)$-norm and the $L^{p+1}(\R^2)$-norm, respectively. We define the action
\[  S_{\lambda, \varepsilon} (v)
    :=\frac12Q_{\lambda, \varepsilon}(v)
    -\frac{1}{p+1}N_{\varepsilon}(v)  \]
and the Nehari functional      
\[  K_{\lambda, \varepsilon}(v)
    :=\langle S_{\lambda, \varepsilon}'(v), v\rangle 
    =Q_{\lambda, \varepsilon}(v)
    -N_{\varepsilon}(v)  \]
for $v\in H_\alpha^1(\R^2)$.
We consider the minimization problem
\begin{equation} \label{eq:n4.3}
    \mu_{\lambda, \varepsilon}
    :=\inf\{S_{\lambda, \varepsilon} (v)\:\!|\; 
    v\in H_{\alpha, \rad}^1(\R^2)\setminus\{0\},\: K_{\lambda, \varepsilon}(v)=0\}.
\end{equation}
Note that $\mu_{\lambda, \varepsilon}$ can rewritten as 
\begin{align}\label{eq:n4.4}
    \mu_{\lambda, \varepsilon}
    =\inf\set{c_p Q_{\lambda, \varepsilon}(v)}{v\in H_{\alpha, \rad}^1(\R^2)\setminus\{0\},\: K_{\lambda, \varepsilon}(v)=0},\quad 
    c_p:=\frac{p-1}{2(p+1)}.
\end{align}

\begin{lemma} \label{lem:Knega}
If $c_p Q_{\lambda, \varepsilon}(v)\le \mu_{\lambda, \varepsilon}$ and $c_p N_{\varepsilon}(v)\ge \mu_{\lambda, \varepsilon}$, then $v$ is a minimizer for \eqref{eq:n4.4}. 
\end{lemma}

\begin{proof}
Note that $\mu_{\lambda, \varepsilon}>0$ because for $v\in H_{\alpha, \rad}^1(\R^2)\setminus\{0\}$ with $K_{\lambda, \varepsilon}(v)=0$ we have the estimate 
\begin{align*}
    &Q_{\lambda, \varepsilon}(v)
    =N_{\varepsilon}(v)
    \lesssim Q_{\lambda, \varepsilon}(v)^{(p+1)/2},&
    \text{i.e., }
    1\lesssim Q_{\lambda, \varepsilon}(v)^{(p-1)/2}. 
\end{align*}
Thus, the assumption implies $K_{\lambda, \varepsilon}(v)\le 0$ and $v\ne 0$. From the shape of the function $\sigma\mapsto K_{\lambda, \varepsilon}(\sigma v)$, there exists $\sigma\in(0, 1]$ such that $K_{\lambda, \varepsilon}(\sigma v)=0$. Thus, from \eqref{eq:n4.4} we obtain
\[  c_pQ_{\lambda, \varepsilon}(v)
    \le \mu_{\lambda, \varepsilon}
    \le c_pQ_{\lambda, \varepsilon}(\sigma v)
    =\sigma^2c_pQ_{\lambda, \varepsilon}(v)
    \le c_pQ_{\lambda, \varepsilon}(v).    \]
This implies $c_pQ_{\lambda, \varepsilon}(v)=\mu_{\lambda, \varepsilon}$, $\sigma=1$, and $K_{\lambda, \varepsilon}(v)=0$. 
\end{proof}

\begin{lemma} \label{lem:minim}
There exists a minimizer $v_\infty\in H_{\alpha, \rad}^1(\R^2)$ for \eqref{eq:n4.4}.
\end{lemma}

\begin{proof}
Let $(v_n)_{n=1}^\infty\subset H_{\alpha, \rad}^1(\R^2)$ be a minimizing sequence of $\mu_{\lambda, \varepsilon}$. Then since 
\begin{align*}
    c_pQ_{\lambda, \varepsilon}(v_n)
    =c_pN_{\varepsilon}(v_n)    
    \to\mu_{\lambda, \varepsilon}
\end{align*}
as $n\to\infty$, there exist a subsequence $(v_{n_j})_{j=1}^\infty$ and $v_\infty\in H_{\alpha, \rad}^1(\R^2)$ such that $v_{n_j}\rightharpoonup v_\infty$ weakly in $H_\alpha^1(\R^2)$. In particular, $c_pQ_{\lambda, \varepsilon}(v_\infty)\le \lim_{j\to\infty}c_pQ_{\lambda, \varepsilon}(v_{n_j})=\mu_{\lambda, \varepsilon}$. Moreover, by the radial compactness lemma, we obtain $N_{\varepsilon}(v_\infty)=\mu_{\lambda, \varepsilon}$. Therefore,
by Lemma~\ref{lem:Knega}, $v_\infty$ is the desired function.
\end{proof}

To prove the existence of a nonnegative minimizer for \eqref{eq:n4.4}, we need the following the fundamental solution of $-\Delta+g_{\lambda, \varepsilon}(|x|)$.

\begin{lemma} \label{lem:5.4}
The fundamental solution $G_{\lambda, \varepsilon}\in L_{\rad}^2(\R^2)\cap C^2(\R^2\setminus\{0\})$ of $-\Delta+g_{\lambda, \varepsilon}(|x|)$ exists, that is $(-\Delta+g_{\lambda, \varepsilon}(|x|))G_{\lambda, \varepsilon}=\delta_0$, which satisfies the following.
\begin{itemize}
\item $G_{\lambda, \varepsilon}(x)=G_\lambda(x)-k_{\lambda, \varepsilon} +o(1)$ as $x\to 0$ for some $k_{\lambda, \varepsilon}\in\R$.
\item $G_{\lambda, \varepsilon}$ is positive in $\R^2\setminus\{0\}$.
\item $G_\lambda-G_{\lambda, \varepsilon}\in H^2(\R^2)$. 
\end{itemize}
\end{lemma}

\begin{proof}
Let $\widetilde{G}_{\lambda, \varepsilon}(r)$ be the solution of $-\widetilde{G}''-\widetilde{G}'/r+g_{\lambda, \varepsilon}(r)\widetilde{G}=0$ with the conditions $\widetilde{G}(3)=G_\lambda(3)$ and $\widetilde{G}'(3)=G'_\lambda(3)$. By $g_{\lambda, \varepsilon}=\lambda$ on $[3, \infty)$ and uniqueness of the solutions of the Cauchy problem, we have $\widetilde{G}_{\lambda, \varepsilon}=G_\lambda$ on $[3, \infty)$. Note that $\widetilde{G}_{\lambda, \varepsilon}$ is decreasing. Indeed, if not, there exists $r_0\in(0, 3)$ such that $\widetilde{G}_{\lambda, \varepsilon}(r_0)>0$, $\widetilde{G}_{\lambda, \varepsilon}'(r_0)=0$, and $\widetilde{G}_{\lambda, \varepsilon}''(r_0)\le 0$, which contradicts $-\widetilde{G}''(r_0)-\widetilde{G}'(r_0)/r_0+g_{\lambda, \varepsilon}(r_0)\widetilde{G}(r_0)=0$. In particular $\widetilde{G}_{\lambda, \varepsilon}$ is positive on $(0, \infty)$. Let $F_\lambda(r)$ be the solution of 
\begin{equation}\label{eq:HelODE}
    -F''-\frac{1}{r}F'+\lambda F=0,\quad r>0
\end{equation} 
with the initial conditions $F(0)=1$ and $F'(0)=0$. Note that $F_\lambda(r)=I_0(\sqrt{\lambda}r)$, where $I_0$ is the modified Bessel function of the first kind of order zero (see \cite[Chapter~10]{NIST:DLMF}). Since the pair $(F_\lambda, G_\lambda)$ consists a system of fundamental solutions of \eqref{eq:HelODE} and $g_{\lambda, \varepsilon}=\lambda$ on $(0, 1]$,
we see that 
\begin{equation}\label{eq:lincombGtil}
    \widetilde{G}_{\lambda, \varepsilon}=c_1 G_\lambda+c_2 F_\lambda\quad 
    \text{on $(0, 1]$ for some $c_1, c_2\in\R$}.    
\end{equation}
Since $\widetilde{G}_{\lambda, \varepsilon}$ is positive, we see that $c_1\ge0$. Moreover, since $\widetilde{G}_{\lambda, \varepsilon}$ is decreasing and since $F_{\lambda}$ is positive and increasing, we have $c_1>0$. Now we put 
\begin{align*}
    &G_{\lambda, \varepsilon}(x)
    :=\frac{1}{c_1}\widetilde{G}_{\lambda, \varepsilon}(|x|),
    \quad x\in\R^2,&
    &k_{\lambda, \varepsilon}
    :=-\frac{c_2}{c_1}.  
\end{align*}  
Then we can easily check all of the assertions.
\end{proof}

By using $G_{\lambda, \varepsilon}$ given in Lemma~\ref{lem:5.4}, we can obtain another expression of the operator $-\Delta_\alpha$ for $\alpha\in\R$ as 
\[  \left\{\begin{aligned}
   &D(-\Delta_\alpha)
    =\bset{f+\frac{f(0)}{\alpha+\beta(\lambda)+k_{\lambda, \varepsilon}}G_{\lambda, \varepsilon}}{f\in H^2(\R^2)},
\\ &\begin{multlined}
    (-\Delta_\alpha+g_{\lambda, \varepsilon}(|x|))v
    =(-\Delta+g_{\lambda, \varepsilon}(|x|))f
\\  \text{for }v=f+\frac{f(0)}{\alpha+\beta(\lambda)+k_{\lambda, \varepsilon}}G_{\lambda, \varepsilon}\in D(-\Delta_\alpha).
    \end{multlined}
    \end{aligned}\right. \]
The quadratic form in the operator domain is written for $v=f+f(0)(\alpha+\beta(\lambda)+k_{\lambda, \varepsilon})^{-1}G_{\lambda, \varepsilon}\in D(-\Delta_\alpha)$ as 
\[  Q_{\lambda, \varepsilon}(v) 
    =((-\Delta+g_{\lambda, \varepsilon})f, f)_{L^2}
    +\frac{|f(0)|^2}{\alpha+\beta(\lambda)+k_{\lambda, \varepsilon}}
    \]
Thus, in the form domain we can express the quadratic form for $v=f+cG_{\lambda, \varepsilon}\in H_\alpha^1(\R^2)$ as 
\begin{equation}\label{eq:quaddel}
    Q_{\lambda, \varepsilon}(v) 
    =\|\nabla f\|_{L^2}^2
    +\int_{\R^2}g_{\lambda, \varepsilon}(|x|)|f(x)|^2\,dx 
    +(\alpha+\beta(\lambda)+k_{\lambda, \varepsilon})|c|^2.
\end{equation}

\begin{remark} \label{rem:expc1c2}
Although it is not essential in our argument, we can show that $k_{\lambda, \varepsilon}$ in Lemma~\ref{lem:5.4} is positive. Indeed, $\widetilde{G}_{\lambda, \varepsilon}(r)$ appearing in the proof satisfies
\[  \widetilde{G}_{\lambda, \varepsilon}''
    +\frac{1}{r}\widetilde{G}_{\lambda, \varepsilon}'-\lambda \widetilde{G}_{\lambda, \varepsilon}
    =f,\quad r>0  \]
with $f(r):=\varepsilon\chi(r)u_\lambda(r)^{p-1}\widetilde{G}_{\lambda, \varepsilon}$. By using the method of variation of parameters, we have the explicit formula
\[  \widetilde{G}_{\lambda, \varepsilon}(r) 
    =\Bigl(1-2\pi\int_3^r s F_\lambda(s)f(s)\,ds \Bigr)G_\lambda(r) 
    +2\pi\int_3^r s G_\lambda(s)f(s)\,ds \, F_\lambda(r). \]
Comparing this and \eqref{eq:lincombGtil}, we see that 
\begin{align*}
    c_1&=1+2\pi\int_1^3 s F_\lambda(s)f(s)\,ds>0,
\\  c_2&=-2\pi\int_1^3 s G_\lambda(s)f(s)\,ds<0.
\end{align*}
Thus, $k_{\lambda, \varepsilon}=-c_2/c_1>0$. 
\end{remark}

\begin{lemma} \label{lem:5.7}
There exists a nonnegative minimizer $v_{\lambda, \varepsilon}\in H_{\alpha, \rad}^1(\R^2)$ for \eqref{eq:n4.3}.
\end{lemma}

\begin{proof}
Let $v_\infty$ be a minimizer for \eqref{eq:n4.3} obtained in Lemma \ref{lem:minim} and be decomposed as $v_\infty=f_\infty+c_\infty G_{\lambda, \varepsilon}$ with $f_\infty\in H^1(\R^2)$ and $c_\infty\in\R$. Put $v_{\lambda, \varepsilon}:=|f_\infty|+|c_\infty|G_{\lambda, \varepsilon}\in H_\alpha^1(\R^2)$. Then by the expression \eqref{eq:quaddel} we have $c_p Q_{\lambda, \varepsilon}(v_{\lambda, \varepsilon})\le c_p Q_{\lambda, \varepsilon}(v_\infty)=\mu_{\lambda, \varepsilon}$. From the positivity of $G_{\lambda, \varepsilon}$, we also have $N_{\lambda, \varepsilon}(v_{\lambda, \varepsilon})\ge N_{\lambda, \varepsilon}(v_\infty)=\mu_{\lambda, \varepsilon}$. Thus, Lemma \ref{lem:Knega} implies that $v_{\lambda, \varepsilon}$ is a minimizer for \eqref{eq:n4.3}. 
\end{proof}

We define the linearized operator $L_{\lambda, \varepsilon}^+$ around $v_{\lambda, \varepsilon}$ as 
\[  L_{\lambda, \varepsilon}^+w
    :=(-\Delta_\alpha+g_{\varepsilon, \lambda}(|x|))w-ph_{\lambda, \varepsilon}(|x|)v_{\lambda, \varepsilon}(x)^{p-1}w \]
for $w\in H_\alpha^1(\R^2; \R)$.

\begin{lemma} \label{lem:posiK'}
If $w\in H_{\alpha, \rad}^1(\R^2;\R)$ satisfies $\langle K_{\lambda, \varepsilon}'(v_{\lambda, \varepsilon}), w\rangle =0$, then $\langle L_{\lambda, \varepsilon}^+w, w\rangle \ge 0$.
\end{lemma}

\begin{proof}
Since $K_{\lambda, \varepsilon}(v_{\lambda, \varepsilon})=0$ and
\begin{align*}
    &\partial_tK_{\lambda, \varepsilon}(v_{\lambda, \varepsilon}+sw+tv_{\lambda, \varepsilon})|_{(s, t)=(0, 0)}
    =\langle K_{\lambda, \varepsilon}'(v_{\lambda, \varepsilon}), v_{\lambda, \varepsilon}\rangle
    = -(p-1)N_{\lambda, \varepsilon}(v_{\lambda, \varepsilon})<0,
\end{align*}
the implicit function theorem implies that there exists a function $t(s)$ defined around $s=0$ such that $t(0)=0$ and $K_{\lambda, \varepsilon}(v_{\lambda, \varepsilon}+sw+t(s)v_{\lambda, \varepsilon})=0$. By the characterization \eqref{eq:n4.3}, the function $s\mapsto S_{\lambda, \varepsilon}(v_{\lambda, \varepsilon}+sw+t(s)v_{\lambda, \varepsilon})$ has the minimum at $s=0$ and satisfies
\[  0\le \partial_s^2S_{\lambda, \varepsilon}(v_{\lambda, \varepsilon}+sw+t(s)v_{\lambda, \varepsilon})|_{s=0}
    =\langle L_{\lambda, \varepsilon}^+w, w\rangle, \]
where we used $S_{\lambda, \varepsilon}'(v_{\lambda, \varepsilon})=0$ and $t'(0)=-
\langle K_{\lambda, \varepsilon}'(v_{\lambda, \varepsilon}), w\rangle/\langle K_{\lambda, \varepsilon}'(v_{\lambda, \varepsilon}), v_{\lambda, \varepsilon}\rangle=0$. 
\end{proof}

\begin{proof}[Proof of Proposition~\ref{prop:nondrad}]
First, we show $u_\lambda=v_{\lambda, \varepsilon}$ for small $\varepsilon>0$. Since $v_{\lambda, \varepsilon}$ is a minimizer for \eqref{eq:n4.3} and $\langle K_{\lambda, \varepsilon}'(v_{\lambda, \varepsilon}), v_{\lambda, \varepsilon}\rangle \ne0$, the Lagrange multiplier theorem implies that $S_{\lambda, \varepsilon}'(v_{\lambda, \varepsilon})=lK_{\lambda, \varepsilon}'(v_{\lambda, \varepsilon})$ for some $l\in\R$. Then we see that $l=0$ because 
\[  0=\langle S_{\lambda, \varepsilon}'(v_{\lambda, \varepsilon}), v_{\lambda, \varepsilon}\rangle 
    =l\langle K_{\lambda, \varepsilon}'(v_{\lambda, \varepsilon}), v_{\lambda, \varepsilon}\rangle \]
and $\langle K_{\lambda, \varepsilon}'(v_{\lambda, \varepsilon}), v_{\lambda, \varepsilon}\rangle\ne0$. Therefore, $v_{\lambda, \varepsilon}$ is a solution of  \eqref{eq:elldel}, and the corresponding boundary value problem of the ordinary differential equation is 
\begin{equation} \label{eq:ODEdel}
    \begin{dcases}
    - v''
    -\frac{1}{r} v' 
    +g_{\lambda, \varepsilon}(r)v 
    -h_\varepsilon(r)v^p 
    =0,& 
    r>0,
\\  v(r)
    =y_v\Bigl(-\frac{1}{2\pi}\log r
    +\alpha\Bigr)+o(1)&
    \text{as }r\downarrow 0
    \text{ for some }y_v>0,
\\  v(r)\to 0&
    \text{as }r\to\infty.
    \end{dcases}
\end{equation} 
Since $v_{\lambda, \varepsilon}\ge 0$, the uniqueness of the Cauchy problem of second order ordinary differential equations implies $v_{\lambda, \varepsilon}>0$. 

In order to use the Poho\v{z}aev identity, let $a_\varepsilon$, $b_\varepsilon$, $c_\varepsilon$, and $C_\varepsilon$ be the functions defined in \eqref{eq:a(r)}--\eqref{eq:C(r)} with $d=2$, $g(r)=g_{\lambda, \varepsilon}$, and $h(r)=h_\varepsilon(r)$. Since $C(r)>0$ for all $r>0$, we can take $\varepsilon>0$ small so that $C_\varepsilon(r)>0$ for all $r>0$. Then by the same argument as in Sections~\ref{sec:2} and \ref{sec:3}, we see that \eqref{eq:ODEdel} has at most one positive solution. Since $u_\lambda$ and $v_{\lambda, \varepsilon}$ are positive solutions of \eqref{eq:ODEdel}, we obtain $u_\lambda=v_{\lambda, \varepsilon}$. 
    
To prove the assertion, let $w\in H_{\alpha, \rad}^1(\R^2;\R)$ and $L_{\lambda}^+ w =0$. If we take $\varepsilon>0$ small so that $u_\lambda=v_{\lambda, \varepsilon}$, then we can write 
\begin{align} \notag
    L_{\lambda, \varepsilon}^+
   &=-\Delta_\alpha+\lambda +\varepsilon\chi(r)u_\lambda(r)^{p-1}-p(1+\varepsilon\chi(r))v_{\lambda, \varepsilon}^{p-1}
\\ \label{eq:Ldeldif} 
   &=L_{\lambda}^+-(p-1)\varepsilon\chi u_\lambda^{p-1}.
\end{align}
Since $\langle K_{\lambda, \varepsilon}'(v_{\lambda, \varepsilon}), v_{\lambda, \varepsilon}\rangle \ne 0$, there exists $s\in\R$ such that $\langle K_{\lambda, \varepsilon}'(v_{\lambda, \varepsilon}), w+sv_{\lambda, \varepsilon}\rangle=0$. Lemma~\ref{lem:posiK'} implies $\langle L_{\lambda, \varepsilon}^+(w+sv_{\lambda, \varepsilon}), w+sv_{\lambda, \varepsilon}\rangle\ge 0$. Using \eqref{eq:Ldeldif}, we obtain 
\begin{align*}
    0 
   &\ge -\langle L_{\lambda, \varepsilon}^+w, w\rangle-2s\langle L_{\lambda, \varepsilon}^+w, v_{\lambda, \varepsilon}\rangle
    -s^2\langle L_{\lambda, \varepsilon}^+v_{\lambda, \varepsilon}, v_{\lambda, \varepsilon}\rangle
\\ &=\varepsilon(p-1)\int_{\R^2}\chi u_\lambda^{p-1}(w+su_\lambda)^2\,dx 
    +s^2(p-1)\int_{\R^2}u_\lambda^{p+1}\,dx.
\end{align*}
Since the each term in the last expression is nonnegative, we have $s=0$, and thus, $w=0$ on $\operatorname{supp}\chi$ due to $\varepsilon>0$. Therefore, the uniqueness of the Cauchy problem for the equation $L_\lambda^+w=0$, we obtain $w=0$. This means the assertion.
\end{proof}

\section{Nondegeneracy in the whole function space}\label{sec:NDwhole}

In this section, we complete the proof of Theorem~\ref{thm:2}. Let $(\mu_j)_{j=0}^\infty$ be the increasing sequence consisting of all eigenvalues of the Laplace--Beltrami operator $-\Delta_{S^1}$ on $L^2(S^1)$ and let $(e_j)_{j=0}^\infty$ be their corresponding normalized eigenfunctions. It is know that 
\begin{equation*}
    0=
    \mu_0
    <\mu_1=\mu_2 =1 
    <\mu_3\le \cdots,
\end{equation*}
and $(e_j)_{j=0}^\infty$ is a complete orthogonal basis of $L^2(S^1)$. 

Let $w\in H_\alpha^1(\R^2; \R)$ satisfy $L_\lambda^+w=0$. Note that $w$ is belongs to $D(-\Delta_\alpha)$ from the equation $L_\lambda^+w=0$ and can be decomposed as $w=f+f(0)(\alpha+\beta(\lambda))^{-1}G_\lambda$ with some $f\in H^2(\R^2)$. We put
\[  w_j(x) 
    :=\int_{S^1}w(|x|\omega)e_j(\omega)\,d\omega
    \]
for $j\in\{0\}\cup\mathbb{N}$. Obviously $w_j$ is radially symmetric, and we can extend $w$ as 
\[  w(x)=\sum_{j=0}^\infty w_j(|x|)e_j\Bigl(\frac{x}{|x|}\Bigr). \]

\begin{lemma} \label{lem:5.1}
For each $j\in\{0\}\cup\mathbb{N}$, $w_j$ satisfies the following.
\begin{itemize}
\item $w_j$ belongs to $D(-\Delta_\alpha)$.
\item $\lim_{x\to 0}w_j(x)=0$ if $j\ge 1$. 
\item $w_j$ is in $C^2(0, \infty)$ and solves the equation
\begin{equation}\label{eq:wj}
    -w_j''
    -\frac{1}{r}w_j' 
    +\Bigl(\frac{\mu_j}{r^2}
    +\lambda 
    -pu_\lambda^{p-1} 
    \Bigr)w_j
    =0,\quad 
    r>0.
\end{equation}
\end{itemize}
\end{lemma}

\begin{proof}
By the decomposition $w=f+f(0)(\alpha+\beta(\lambda))^{-1}G_\lambda$ with $f\in H^2(\R^2)$ and 
\[  w_j(x) 
    =\int_{S^1}f(|x|\omega)e_j(\omega)\,d\omega
    +\frac{f(0)}{\alpha+\beta(\lambda)}\Bigl(\int_{S^1}e_j(\omega)\,d\omega\Bigr) G_\lambda(x),
    \]
we see that $w_j\in D(-\Delta_\alpha)$. Since $e_0$ is a nonzero constant and $(e_0, e_j)_{L^2(S^1)}=\delta_{0j}$, where $\delta_{0j}$ is the Kronecker delta, we have $\int_{S^1}e_j(\omega)\,d\omega=0$ for $j\ge 1$. This means that if $j\ge 1$, then
\[  w_j(x) 
    =\int_{S^1}f(|x|\omega)e_j(\omega)\,d\omega \]
and $\lim_{x\to 0}w_j(x)=0$.  The equation $L_\lambda^+ w=0$ away from the origin can be rewritten as 
\begin{align*}
  0&=(-\Delta_\alpha+\lambda)w 
    -pu_\lambda^{p-1}w
\\ &=(-\Delta+\lambda)w 
    -pu_\lambda^{p-1}w
\\ &=\Bigl(-\partial_r^2
    -\frac{1}{r}\partial_r 
    -\frac{1}{r^2}\Delta_{S^1}
    +\lambda\Bigr)w
    -pu_\lambda^{p-1}w,\quad 
    x\in\R^2\setminus\{0\}.
\end{align*}
Thus, we see that $w_j$ satisfies \eqref{eq:wj} in weak sense. Moreover, we see from this equation that $w_j$ belongs to $C^2(0, \infty)$ and satisfies \eqref{eq:wj} in the classical sense. This completes the proof.
\end{proof}

\begin{lemma}
For each $j\in\{0\}\cup\mathbb{N}$, $w_j$ satisfies
\begin{equation}\label{eq:5.3}
    \int_s^t\frac{1}{r}(1-\mu_j)u_\lambda'(r)w_j(r)\,dr 
    =\xi_j(t)-\xi_j(s)
\end{equation}
for all $0<s<t$, where 
\begin{equation}\label{eq:5.4}
    \xi_j(r) 
    :=r\bigl(u_\lambda''(r) w_j(r)-u_\lambda'(r) w_j'(r)\bigr)
    .
\end{equation}
\end{lemma}

\begin{proof}
Differentiating \eqref{eq:ODE1} we have 
\begin{equation*}
    -u'''-\frac{1}{r}u''+\frac{1}{r^2}u' 
    +\lambda u'-pu^{p-1}u'
    =0,\quad 
    r>0.
\end{equation*}
By this equation and \eqref{eq:wj} we have 
\begin{align*}
    \big(r(u''w_j-u'w_j')\big)'
   &=u''w_j-u'w_j' 
    +r(u'''w_j-u'w_j'')
\\ &=\Big(\frac{1}{r^2}u' 
    +\lambda u' 
    -pu^{p-1}u'\Big) rw_j
    -\Big(\frac{\mu_j}{r^2}+\lambda -pu^{p-1}\Big)u'w_j
\\ &=\frac{1}{r}(1-\mu_j)u'w_j. 
\end{align*}
Integrating this over $[s, t]$ we obtain the conclusion. 
\end{proof}

\begin{proof}[Proof of Theorem~\ref{thm:2}]
It suffices to show that $w_j= 0$ for all $j\ge 0$. We divide the proof into three cases.

\noindent\textbf{Case 1:} $\mu_j=0$, i.e., $j=0$. 
$w_0= 0$ follows from Lemma~\ref{eq:wj} and Proposition~\ref{prop:nondrad}.

\noindent\textbf{Case 2:} $\mu_j=1$, i.e., $j=1, 2$. 
By \eqref{eq:5.3} we have 
$
    \xi_j(s)
    =\xi_j(t)
$
for all $0<s<t$. Noting that $\xi_j(t)\to 0$ as $t\to\infty$, we have $\xi_j\equiv0$. This and \eqref{eq:5.4} imply that $w_j$ solves the differential equation 
\[  u_\lambda''(r)w_j-u_\lambda'(r)w_j'=0,\quad 
    r>0.    \]
Solving it we obtain 
\[  w_j(r)=Cu_\lambda'(r),\quad r>0 \]
for some constant $C\in\R$. Since $w_j(0) =0$ and $u_\lambda'(r)=y_{u_\lambda}(2\pi r)^{-1}+o(1)$ as $r\to 0$, we have $C=0$. Thus, $w_j=0$.

\noindent\textbf{Case 3:} $\mu_j>1$ i.e., $j\ge 3$. 
Suppose that $w_j\not\equiv 0$ for some $j\ge 3$. Note that by Lemma~\ref{lem:5.1}
\[  \lim_{r\downarrow 0}w_j(r)=0\quad \text{and}\quad 
    \lim_{r\to\infty}w_j(r)=0.  \]
Thus, by the continuity of $w_j$ (and by replacing $w$ with $-w$ if necessary) we can take a interval $(s, t)\subset(0, \infty)$ such that 
\begin{align*}
    &w_j(r)>0,\quad 
    r\in(s, t),&
    &\lim_{r\downarrow s}w_j(r)
    =\lim_{r\uparrow t}w_j(r)
    =0,
\end{align*}
and we also see that 
\begin{align*}
    &\lim_{r\downarrow s}w_j'(r)
    \in[0, \infty],&
    &\lim_{r\uparrow t}w_j'(r)\le 0.
\end{align*}
By \eqref{eq:5.3} and $\mu_j>1$, we have 
\begin{equation} \label{eq:6.7}
    -\infty\le \lim_{r\downarrow s}\xi_j(r)
    <\lim_{r\uparrow t}\xi_j(r).
\end{equation} 
On the other hand we see that 
\begin{align*}
    \lim_{r\uparrow t}\xi_j(r) 
   &=\lim_{r\uparrow t}\bigl(-u_\lambda'(r)+\lambda ru_\lambda(r)-ru_\lambda(r)^{p}\bigr)w_j(r) 
    -ru_\lambda'(r) w_j'(r)
\\ &=\left\{\begin{alignedat}{2}
   &\mathopen{}-tu_\lambda'(t)w_j'(t)\le 0 & \quad
   &\text{if }t<\infty,
\\ & \omit\hfil$0$\hfil &  
   &\text{if }t=\infty,
    \end{alignedat}\right.
\end{align*}
and 
\begin{align*}
    \lim_{r\downarrow s}\xi_j(r) 
   &=\lim_{r\downarrow s}\bigl(-u_\lambda'(r)+\lambda ru_\lambda(r)-ru_\lambda(r)^{p}\bigr)w_j(r) 
    -ru_\lambda'(r) w_j'(r)
\\ &=\left\{\begin{alignedat}{2}
   &\mathopen{}-su_\lambda'(s)w_j'(s)\ge 0 & 
   &\text{if }s>0,
\\ &\frac{y_{u_\lambda}}{2\pi}\lim_{r\downarrow s}\Bigl(\frac{w_j(r)}{r}
    +w_j'(r)\Bigr)
    \in [0, \infty] & \quad
   &\text{if }s=0.
    \end{alignedat}\right.
\end{align*}
Thus, we have $+\infty\ge\lim_{r\downarrow s}\xi_j(r)\ge 0 \ge\lim_{r\uparrow t}\xi_j(r)$. This contradicts \eqref{eq:6.7}. This completes the proof.
\end{proof}

\section*{Acknowledgments}

The author was supported by JSPS KAKENHI Grant Number JP20K14349.
The author would like to thank Mario Rastrelli for discussion on Remark~\ref{rem:expc1c2}.



\bibliographystyle{amsplain_abbrev_nobysame_nonumber}
\bibliography{bibfile}

\end{document}